\documentclass[a4paper,12pt]{article}
\usepackage{amsmath,amsthm, amssymb}

\usepackage{graphicx}

\newtheorem{theorem}{Theorem}[section]
\newtheorem{proposition}{Proposition}[section]

\newtheorem{remark}{Remark}[section]

\begin{document}
\title{Iwasawa invariants and exceptional pairs of certain abelian fields}
\author{\normalsize 
{\large 
Kazato {\sc Nada}, 
Haruma {\sc Sasaki}
and Hiroki {\sc Sumida-Takahashi}
\footnote{\noindent ORCID:0000-0003-3928-9481 The author was partially supported by JSPS KAKENHI Grant Number JP25K06934 and JP26K00600. }}
\\
\\
\normalsize \it Department of Mathematical Sciences, \\ 
\normalsize \it Tokushima University, \\
\normalsize \it Minamijosanjima-cho 2-1, Tokushima 770-8506, JAPAN\\
\normalsize \it e-mail address $:$ hirokit@tokushima-u.ac.jp\\
}
\date{}
\maketitle  
\begin{abstract}
\vspace{.3cm}
{\normalsize
Let $p$ be an odd prime number and $\zeta_p$ a primitive $p$-th root of unity. 
By computing special arithmetic elements modulo prime ideals, 
we have investigated the Iwasawa invariants and exceptional pairs 
$(p,\chi \omega_p^k)$ of $\mathbf{Q}(\sqrt{d},\zeta_p)$. 
Here $\chi$ denotes the Dirichlet character associated to 
$\mathbf{Q}(\sqrt{d})$, and $\omega_p$ the Teichm\"uller character at $p$. 
First, we determine $\varLambda$-isomorphism classes of Iwasawa modules, 
and describe structural differences between exceptional and 
non-exceptional pairs in terms of unramified extensions outside $p$. 
Next, we report 15 new exceptional pairs in the range $|d|<200$ (resp.~$|d|<10$) 
and $p <2,\!000,\!000$ (resp.~$p <30,\!000,\!000$), 
including $(p,k,d)=(28,\!679,\!999,\;2,\!284,\!521,-8)$ 
for which the $\chi \omega_p^{1-k}$-part of the Iwasawa $\lambda$-invariant 
is equal to two. 
\\[.3cm]   
\noindent
Key words: Iwasawa invariant, exceptional pair, unramified extension, ideal class group\\
2020 Mathematics Subject Classification: Primary 11R23; Secondary 11R18, 11R29, 11R70
}
\end{abstract}

\def\Gal{\mbox{\rm Gal}}
\def\ch{\mbox{\rm char$_\varLambda$}}
\def\rk{\mbox{\rm rank}}
\def\rnk{\mbox{\rm -rank}}
\def\Cok{\mbox{\rm Coker}}
\def\od{\mbox{\rm ord}}
\def\Kn{\mbox{\rm Ker}}
\def\Ig{\mbox{\rm Im}}
\def\md{\mbox{\rm \;mod }}
\def\Hom{\mbox{\rm Hom}}
\def\disp{\displaystyle}

\section{\large Introduction}
 
Let $\chi$ be either the trivial character or a primitive quadratic Dirichlet character of conductor $f=f_\chi$, 
and let $p$ be an odd prime number such that $p \nmid f$. 
We fix an embedding $\overline{\mathbf{Q}}$ to $\mathbf{C}_p$. 
Define $d=d_\chi=\chi(-1)f_\chi$, $K=\mathbf{Q}(\sqrt{d},\zeta_{p})$, 
and $K_n=\mathbf{Q}(\sqrt{d},\zeta_{p^{n+1}})$, where $\zeta_m$ denotes a primitive $m$-th root of unity. 
Let $A_n$ be the $p$-part of the ideal class group of $K_n$. 

Set $K_\infty=\cup_{n \geq 0} K_n$, 
$\Gamma=\Gal(K_\infty/K)$, $\Delta=\Gal(K_\infty/\mathbf{Q}_\infty)
\simeq \Gal(K_0/\mathbf{Q})$ and define 
$e_{\varphi}=\frac{1}{\sharp \Delta} \sum_{\delta \in \Delta}
\varphi(\delta)\delta^{-1}$ for a character $\varphi$ of $\Delta$. 

For a $\mathbf{Z}_p[\Delta]$-module $A$, We denote by $A^\varphi$ the $\varphi$-part $e_\varphi A$. 
Let $\lambda_p(\varphi)$, $\mu_p(\varphi)$ and $\nu_p(\varphi)$ 
be the Iwasawa invariants associated to $A_n^{\varphi}$, i.e., 
$$
\sharp A_n^{\varphi}=p^{\lambda_p(\varphi) n+\mu_p(\varphi) p^n+\nu_p(\varphi)}
$$
for sufficiently large $n$. 
By Ferrero-Washington's theorem, we have $\mu_p(\varphi)=0$ for all $p$ and $\varphi$. 

We put $f_0=fp$ and identify $\Delta$ with a subquotient of 
$(\mathbf{Z}/ f_0 \mathbf{Z})^\times$ in the usual way. 
Then, we can identify $\varphi$ with a $\mathbf{Q}_p$-valued Dirichlet character. 
In the following, we assume that a character $\psi$ of $\Delta$ 
is identified with an even primitive Dirichlet character of conductor $f_0$. 
The Iwasawa polynomial $g_{\psi}(T) 
\in \mathbf{Z}_p[T]$ for the $p$-adic $L$-function is defined as follows. 
Let $L_p(s,\psi)$ be the $p$-adic $L$-function constructed 
by \cite{Kubota-Leopoldt}. 
By \cite[\S6]{Iwasawa2}, 
there uniquely exists a unique power series $G_{\psi}(T) \in \mathbf{Z}_p[[T]]$ 
satisfying 
$$
G_{{\psi}}((1+f_0)^{1-s}-1)=L_p(s,{\psi})
$$
for all $s \in \mathbf{Z}_p$ with $\psi \neq \chi^0$. 
By \cite{Ferrero-Washington}, $p$ does not divide $G_{\psi}(T)$. 
By the $p$-adic Weierstrass preparation theorem, 
we can uniquely write $G_{\psi}(T)=g_{\psi}(T)u_{\psi}(T)$, 
where $g_{\psi}(T)$ is a distinguished polynomial 
of $\mathbf{Z}_p[T]$ 
and $u_{\psi}(T)$ is an invertible element of $\mathbf{Z}_p[[T]]$. 
Similarly we can define $g^*_\psi(T) \in \mathbf{Z}_p[T]$ 
from $G^*_{\psi}(T) \in \mathbf{Z}_p[[T]]$
satisfying 
$$
G^*_{\psi}((1+f_0)^{s}-1)=L_p(s,{\psi}).
$$
Put 
$$
\tilde{\lambda}_p(\psi)=\deg g_{\psi}(T)=\deg g^*_{\psi}(T). 
$$
Put $f_n=f_0 p^n$ 
and let 
$\gamma \in \Gamma  \simeq \Gal(\cup_{n \geq 0} \mathbf{Q}(\zeta_{f_n})/\mathbf{Q}(\zeta_{f_0})) $ 
be the generator of $\Gamma$ such that 
$\zeta_{f_n}^{\tilde{\gamma}}=\zeta_{f_n}^{1+f_0}$ for all $n \ge 0$. 
As usual, we can identify the complete group ring 
$\mathbf{Z}_p[[\Gamma]]$ with the formal power series ring 
$\varLambda=\mathbf{Z}_p[[T]]$ by $\gamma =1+T$. 
By this identification, we can consider a 
$\mathbf{Z}_p[[\Gamma]]$-module 
as a $\varLambda$-module. 
For a finitely generated torsion $\varLambda$-module $A$, 
we define the Iwasawa polynomial $\ch(A)$ 
to be the characteristic polynomial of the 
action $T$ on $A \otimes \mathbf{Q}_p$ 
(cf.~\cite[\S 13]{Washington}). 
Let $L_n$ be the maximal unramified abelian extension of $K_n$ 
and $M_n$ the maximal abelian extension of $K_n$ unramified outside $p$. 
By the class field theory, we have 
$
A_n \simeq \Gal(L_n/K_n). 
$
Set 
$L_\infty=\cup_{n\geq 0} L_n$, 
$M_\infty=\cup_{n\geq 0} M_n$, 
$X_\infty=\Gal(L_\infty/K_\infty)$ and 
$Y_\infty=\Gal(M_\infty/K_\infty)$. 
By the Iwasawa main conjecture proved by \cite{Greither, Mazur-Wiles}, 
$$
\ch(X_\infty^{\psi^{-1}\omega})=g_{\psi}^*(T)\quad {\rm and }\quad \ch(Y_\infty^{\psi})=g_{\psi}(T).
$$ 

In the following, we assume that
$$
\psi=\chi \omega^k {\rm \; is\; even,\; and \;} \psi^*=\psi^{-1} \omega=\chi \omega^{p-k} {\rm \; is \; odd}
$$
with $2 \leq k \leq p-2$. Since $p$ does not divide $f$,  
$$
\mbox{(C)} \qquad\qquad\qquad 
\psi(p) \neq 1 \mbox{ and } \psi^*(p) \neq 1. \qquad\qquad 
$$
By (C), we have that 
$A_n^\psi \simeq X_\infty^\psi/\omega_n X_\infty^\psi$ 
and 
$A_n^{\psi^*} \simeq X_\infty^{\psi^*}/\omega_n X_\infty^{\psi^*}$, 
where $\omega_n=(1+T)^{p^n}-1$ (cf.~\cite[Lemma 3 and Remark 4]{Ichimura-Sumida2}). 
Moreover, if $A_0^\psi$ is trivial, we have $\lambda_p(\psi)=\nu_p(\psi)=0$, 
$
X_\infty^\psi=\{0\}, \quad 
Y_\infty^\psi \simeq \varLambda/(g_\psi(T)) \quad 
{\rm and }\quad 
X_\infty^{\psi^*} \simeq \varLambda/(g^*_\psi(T)). 
$
Put 
$$
a_0=a_0(\psi)=L_p(1,\psi)=G_{\psi}(0)\quad {\rm  and } \quad 
b_0=b_0(\psi)=L_p(0,\psi)=G_{\psi}^*(0).
$$ 
Note that $v_p(a_0)=v_p(\sharp \Gal(M_0/K_0)^\psi)$ and $v_p(b_0)=v_p(\sharp \Gal(L_0/K_0)^{\psi^*})$. 

We call $(p,\chi\omega^k)$ an exceptional pair when one of the following conditions holds: 
$$
[\nu]:\; A_0^{\chi \omega^k} \neq \{0\}, \;\; [a_0]:\; v_p(a_0)>1, \;\; 
[b_0]:\; v_p(b_0)>1\;\; {\rm  or}\;\; {\rm [lmd]: }\; \tilde{\lambda}_p(\chi \omega^k)>1.
$$ 

Many authors, including Kummer, Vandiver, D.H. Lehmer, E. Lehmer, Selfridge, Nicol, Pollack, 
Johnson, Wada, Wagstaff, Tanner, Ernvall, Mets\"ankyl\"a, Buhler, Crandall, Sompolski, 
Shokrollahi, Hart, Harvey and Ong, have computed irregular pairs, 
and verified that no exceptional pair exists for $d=1$ and $p<2^{31}=2,\!147,\!483,\!648$ 
(cf.~\cite{Buhler1, Buhler2, Hart}). 

In \cite{Sumida6, Sumida7, Sumida8, Sumida10, Sumida12}, 
by computing special arithmetic elements modulo prime ideals, 
we computed irregular and exceptional pairs in the range $|d|<200$ (resp.~$|d|<10$) 
and $p<1,\!000,\!000$ (resp.~$p <20,\!000,\!000$). 
These data are useful for computing the $p$-part of the ideal class group of $K_n$ 
and higher $K$-groups of the integer ring of $\mathbf{Q}(\sqrt{d})$, 
and for verifying generalized Greenberg's conjecture for $\mathbf{Q}(\sqrt{d}, \zeta_p)$ (cf.~\cite{Sumida13}). 

This paper is organized as follows.
In Section 2, we determine $\varLambda$-isomorphism classes of Iwasawa modules, 
and describe structural differences between exceptional and non-exceptional pairs 
in terms of unramified extensions outside $p$. 
In Section 3, we review algorithms for the computation, 
and report new exceptional pairs obtained by further computation 
in the range $|d|<200$ (resp.~$|d|<10$) 
and $p<2,\!000,\!000$ (resp.~$p <30,\!000,\!000$), 
including $(p,k,d)=(28,\!679,\!999,\;2,\!284,\!521,-8)$ with $\lambda_p(\chi \omega_p^{1-k})=2$. 

\section{\large Unramified abelian extensions outside $p$}

We use the notation in the previous section and assume that $\psi=\chi \omega^k$ satisfies (C). 
Let $D_n$ be the subgroup of $A_n$ which is generated by the class of prime ideals lying above $p$ 
and put $A_n'=A_n/D_n$. Let $E_n$ be the unit group of $K_n$ and $E'_n$ the $p$-unit group of $K_n$.  
By (C), $D_n^\psi=D_n^{\psi^*}=\{0\}$, $A_n^\psi \simeq {A'_n}^\psi$, $A_n^{\psi^*} \simeq {A'_n}^{\psi^*}$ and 
$H^1(K_n/K,E_n)^\psi \simeq H^1(K_n/K,E'_n)^\psi$,where $\psi^*= \psi^{-1} \omega=\chi \omega^{1-k}$. 

In the following, we determine the $\varLambda$-isomorphism classes of Iwasawa modules by 
the Iwasawa main conjecture and \cite[Theorem 12, 17, 18 and Lemma 12]{Iwasawa3}. 
Note that (C) is (A) in \cite{Ichimura-Sumida2}. 
In \cite{Sumida}, we define the set of $\varLambda$-isomorphism classes for a distinguished polynomial 
$f(T) \in \varLambda$:
$$
\mathcal{M}_{f(T)}=\{[M]\;|\; {\rm char}_\varLambda(M)=f(T), 
\;\; M\;{\rm has \; no \; nontrivial \; finite}\; \varLambda-{\rm submodule}\}
$$
and show that $\sharp \mathcal{M}_{f(T)}$ is finite if and only if $f(T)$ is square-free. 
If $(p,\psi)$ does not satisfy $[\nu]$, it is easy to determine the 
$\varLambda$-isomorphism classes of Iwasawa modules, because they are cyclic $\varLambda$-modules. 
If $(p,\psi)$ satisfies $[\nu]$, it is not easy to determine the 
$\varLambda$-isomorphism classes of Iwasawa modules especially when $\deg f(T) \geq 3$ in general. 
When $\deg f(T)=1$ or $f(T)$ is an Eisenstein polynomial, $\sharp \mathcal{M}_{f(T)}=1.$ 
When $\deg f(T)=2$, $\mathcal{M}_{f(T)}$ is explicitly determined in \cite{Koike,Sumida}. 
However, when $\deg f(T) \ge 3$, $\mathcal{M}_{f(T)}$ is not explicitly determined 
except for some cases (cf.~\cite{Franks, Murakami0, Murakami1, Murakami2, Sumida}). 

First, if $(p,\chi \omega^k)$ is a regular pair, we have 
$v_p(a_0(\psi))=0$, $v_p(b_0(\psi))=0$ and 
$\tilde{\lambda}_p(\psi)=\lambda_p(\psi^*)=\nu_p(\psi)=\nu_p(\psi^*)=0$. 
Then, we have the following simple $\varLambda$-module structures of Galois groups. 

\noindent
{\bf Case } $[{\rm R}]$ \\
$$
\begin{array}{lll}
X_\infty^{\psi}=\{0\}, & & Y_\infty^{\psi}=\{0\}, \\
X_\infty^{\psi^*}=\{0\} & \;{\rm and}\; & Y_\infty^{\psi^*} \simeq \varLambda. \\
\end{array}
$$

Next, we assume that $(p,\chi \omega^k)$ is an irregular pair. 
Then, $v_p(a_0(\psi))>0$, $v_p(b_0(\psi))>0$ and $\tilde{\lambda}_p(\psi)=\lambda_p(\psi^*)>0$, 
where $\psi^*=\chi \omega^{1-k}$. 

If $(p,\psi)$ is not an exceptional pair, 

\noindent
{\bf Case } $[{\rm nE}]$ \\
$$
\begin{array}{lll}
X_\infty^{\psi}=\{0\}, & & Y_\infty^{\psi} \simeq \varLambda/(T-\alpha), \\
X_\infty^{\psi^*} \simeq \varLambda/(T-\alpha^*) & \;{\rm and}\; & 
Y_\infty^{\psi^*} \simeq \varLambda \\
\end{array}
$$
for some $\alpha$, $\alpha^* \in p \mathbf{Z}_p \setminus p^2 \mathbf{Z}_p$. 

Finally, we consider four cases that $(p,\psi)$ satisfies only one of conditions 
$[\nu]$, $[a_0]$, $[b_0]$ and $[{\rm lmd}]$, which is usual when $p$ is large. \\

\noindent
{\bf Case } $[\nu]$ \\
\quad (i) If Greenberg's conjecture is true for $p$ and $\psi$, 
$$
\begin{array}{lll}
X_\infty^{\psi} \simeq \varLambda/(T-\alpha, p^{\nu}), & &Y_\infty^{\psi} \simeq \varLambda/(T-\alpha), \\
X_\infty^{\psi^*} \simeq \varLambda/(T-\alpha^*) & 
\;{\rm and}\; & Y_\infty^{\psi^*} \simeq (T,p) \subset \varLambda \\
\end{array}
$$
\quad for some $\alpha$, $\alpha^* \in p \mathbf{Z}_p \setminus p^2 \mathbf{Z}_p$ 
and $\nu \in \mathbf{Z}_{\geq 1}$. \\
\quad (ii) If Greenberg's conjecture is false for $p$ and $\psi$, 
$$
\begin{array}{lll}
X_\infty^{\psi} \simeq \varLambda/(T-\alpha), & & Y_\infty^{\psi} \simeq \varLambda/(T-\alpha), \\
X_\infty^{\psi^*} \simeq \varLambda/(T-\alpha^*) & \;{\rm and}\; & 
Y_\infty^{\psi^*} \simeq \varLambda/(T-\alpha^*)\oplus \varLambda \\
\end{array}
$$
\quad for some $\alpha$, $\alpha^* \in p \mathbf{Z}_p \setminus p^2 \mathbf{Z}_p$. \\

\noindent
{\bf Case } $[a_0]$ \\
$$
\begin{array}{lll}
X_\infty^{\psi}=\{0\}, & & Y_\infty^{\psi} \simeq \varLambda/(T-\alpha), \\
X_\infty^{\psi^*} \simeq \varLambda/(T-\alpha^*) & \;{\rm and}\; & 
Y_\infty^{\psi^*} \simeq \varLambda \\
\end{array}
$$
for some $\alpha \in p^2 \mathbf{Z}_p$ , $\alpha^* \in p \mathbf{Z}_p \setminus p^2 \mathbf{Z}_p$. \\

\noindent
{\bf Case } $[b_0]$ \\
$$
\begin{array}{lll}
X_\infty^{\psi}=\{0\}, & & Y_\infty^{\psi} \simeq \varLambda/(T-\alpha), \\
X_\infty^{\psi^*} \simeq \varLambda/(T-\alpha^*) & \;{\rm and}\; & 
Y_\infty^{\psi^*} \simeq \varLambda \\
\end{array}
$$
for some $\alpha \in p \mathbf{Z}_p\setminus p^2 \mathbf{Z}_p$ , $\alpha^* \in p^2 \mathbf{Z}_p$. \\

\noindent
{\bf Case } $[{\rm lmd}]$ \\
$$
\begin{array}{lll}
X_\infty^{\psi}=\{0\}, & & Y_\infty^{\psi} \simeq \varLambda/(g(T)), \\
X_\infty^{\psi^*} \simeq \varLambda/(g^*(T)) & \;{\rm and}\; & 
Y_\infty^{\psi^*} \simeq \varLambda \\
\end{array}
$$
for some distinguished polynomials $g(T)$, $g^*(T) \in \mathbf{Z}_p[T]$ 
such that $g(0)$, $g^*(0) \in p \mathbf{Z}_p\setminus p^2 \mathbf{Z}_p$, 
$\deg g(T)=\deg g^*(T) \geq 2$. \\

From the above $\varLambda$-structures, we obtain the following structural differences 
of Galois groups over finite extensions in each case. 

\begin{theorem}
Assume that $\psi$ satisfies {\rm (C)} and that 
$(p,\psi)$ is an irregular pair.
In {\rm Case} $[\nu]$, there is a unramified abelian $p$-extension $L'_n$ of $K_n$ such that 
$\Gal(L'_n/K_n)^{\psi}$ is not trivial for $n \geq 0$. 
In {\rm Case} $[a_0]$, there is an unramified abelian $p$-extension $M'_n$ of $K_n$ outside $p$ 
such that the exponent of $\Gal(M'_n/K_n)^{\psi}$ is larger than $p^{n+1}$ for $n \geq 0$. 
In {\rm Case} $[b_0]$, there is an unramified abelian $p$-extension $L'_n$ of $K_n$ 
such that the exponent of $\Gal(L'_n/K_n)^{\psi^*}$ is larger than $p^{n+1}$ for $n \geq 0$. 
In {\rm Case} $[{\rm lmd}]$, there is an unramified abelian $p$-extension $L'_n$ of $K_n$ 
such that the $p$-rank of $\Gal(L'_n/K_n)^{\psi^*}$ is larger than one for $n \geq 1$. 
On the other hand, those extensions do not exist if $(p,\psi)$ is not an exceptional pair.  
\label{thm:1}
\end{theorem}

\begin{proof}
Put $X'=\Gal(L_\infty/K_\infty)^{\psi}$ and $Y'=\Gal(M_\infty/K_\infty)^{\psi^*}$. 
Then, the subgroup of $X'$(resp.~$Y'$) which corresponds to the maximal abelian extension of 
$K_n$ is $\omega_n X'$(resp.~$\omega_n Y'$). 
Since $X' \simeq \Gal(L_n/K_n)^\psi$ and $Y' \simeq \Gal(M_n/K_n)^{\psi^*}$, 
the $\varLambda$-structures of $X'$ and $Y'$ imply the assertions. 
\end{proof}

In the above theorem, we use the $\Delta$-decomposition of Galois groups. 
Let us give another statement which characterizes Case $[{\rm lmd}]$. 
We consider the following two typical cases with $X_\infty \simeq \mathbf{Z}_p^\lambda$. 

\noindent
{\bf Case} $[{\rm lmd} \times {\rm 1}]$
$$
X_\infty^{\psi^*} \simeq \varLambda/(g^*(T)) \quad ({\rm Case\; [{\rm lmd}]}) 
\quad {\rm and } \quad  X_\infty = X_\infty^{\omega^{1-k}}. 
$$

\noindent
{\bf Case} $[{\rm nE}\times \lambda]$
$$
X_\infty^{\chi \omega^{1-k_i}} \simeq \varLambda/(T-\alpha_i^*)\quad ({\rm Case\; [{\rm nE}]}) 
\quad {\rm and } \quad  X_\infty = \bigoplus_{i=1}^\lambda X_\infty^{\omega^{1-k_i}}. 
$$

\begin{theorem}
In {\rm Case} $[{\rm lmd} \times 1]$, there is no unramified abelian $p$-extension $L'_n$ of $K_n$ 
such that 
$$
\Gal(L'_n/K_n)
\simeq (\mathbf{Z}/p^{a_1}\mathbf{Z}) \oplus (\mathbf{Z}/p^{a_2}\mathbf{Z})\oplus
 \cdots \oplus (\mathbf{Z}/p^{a_r}\mathbf{Z})
$$ 
with $r \leq \lambda$ and $a_1 \leq a_2 \leq \ldots \leq a_r$ 
such that $a_r -a_1 \geq 2$ and that $L_n'$ is a Galois extension of $K_0$ for $n \geq 0$. 
On the other hand, in {\rm Case} $[{\rm nE} \times \lambda]$, there exists such an unramified abelian 
extension of $K_n$ which is a Galois extension of $K_0$  for $n \geq 1$.
\label{thm:2}
\end{theorem}

\begin{proof} 
Put $X'=\Gal(L_\infty/K_\infty)^{\psi^*}$. 
Then, the subgroup of $X'$ which corresponds to the maximal abelian extension of $K_n$ is $\omega_n X'$. 
Further, subgroups of $X'$ which correspond to Galois extensions of $K_0$ 
are $\varLambda$-submodules of $X'$. 
We show that a $\varLambda$-submodule $M'$ of $M=\varLambda/(g^*(T))$ is determined by its index 
when $g^*(T)$ is an Eisenstein polynomial, along the proof in \cite[Theorem 2.4]{Sumida11}. 
Put $\lambda=\deg g^*(T) \geq 2$. 
Let $h(T) \in  \varLambda \setminus (g^*(T))$. 
By $p$-adic Weierstrass preparation theorem, $h(T) \equiv p^{s'-1} u(T) f(T) \mod (g^*(T))$, 
where $s' \geq 1$, $u(T) \in \varLambda^\times$, and $f(T)$ is a distinguished polynomial of degree $b' < \lambda$. 
Then, we have 
$(g^*(T),h(T)) =(g^*(T),p^{s'-1} f(T))=(g^*(T),p^{s'-1}f(T),p^{s'-1}g^*(T)-p^{s'-1}T^{\lambda-b'}f(T)) 
=(g^*(T),p^{s'-1}T^{b'},p^{s'})$. 
Let $s$ be the minimum number of such $s'$ contained in $M'$ modulo $(g^*(T))$. 
Further let $b$ be the minimum number of such $b'$ with $s'=s$ contained in $M'$ modulo $(g^*(T))$. 
$M'$ is determined by $s$ and $b$, i.e., $|M/M'|=p^{\lambda (s-1)+b}$. Then, as an abelian group, 
$$
M/M' 
\simeq (\mathbf{Z}/p^{(s-1)}\mathbf{Z})^{\lambda-b} \oplus (\mathbf{Z}/p^{s}\mathbf{Z})^b. 
$$ 
This implies the non-existence of a Galois extension of $K_0$ which satisfies the condition 
in the theorem for $n \geq 0$. 

On the other hand, for $M=\bigoplus_{i=1}^\lambda \varLambda/(T-\alpha_i^*)$, 
$\langle p^{a_1} e_1, p^{a_2} e_2, \ldots, p^{a_\lambda} e_\lambda \rangle$ 
is a $\varLambda$-submodule of $M$ for any $a_i \in \mathbf{Z}_{\geq 0}$, 
where $e_1=(1,0,0, \ldots,0)$, $e_2=(0,1,0,\ldots,0)$, $\ldots$, $e_\lambda=(0,0,\ldots,0,1)$. 
Since $\omega_1 M=p^2 M$, this implies the existence of a Galois extension of $K_0$ which satisfies 
the condition in the theorem for $n \geq 1$. 
\end{proof}

\begin{remark}
When $f(T)$ is an Eisenstein polynomial, the $\varLambda$-isomorphism classes of Galois groups 
are similar to Case $[\nu]$. 
By the above proof, there is only one submodule $(f(T),p^{s-1}T^b,p^s)/(f(T))$ in 
$\varLambda/(f(T))$ of index $p^\nu=p^{\tilde{\lambda}(s-1)+b}$. 

\noindent
{\bf Case } $[\nu]'${ \rm (Eisenstein polynomial)}\\
\quad {\rm (i)} If Greenberg's conjecture is true for $p$ and $\psi$, 
$$
\begin{array}{lll}
X_\infty^{\psi} \simeq \varLambda/(f(T),p^{s-1}T^b,p^s), & &Y_\infty^{\psi} \simeq \varLambda/(f(T)), \\
X_\infty^{\psi^*} \simeq \varLambda/(f^*(T)) & 
\;{\rm and}\; & Y_\infty^{\psi^*} \simeq (T,p) \subset \varLambda \\
\end{array}
$$
\quad for some $b$, $s \in \mathbf{Z}_{\geq 1}$.\\
\quad {\rm (ii)} If Greenberg's conjecture is false for $p$ and $\psi$, 
$$
\begin{array}{lll}
X_\infty^{\psi} \simeq \varLambda/(f(T)), & & Y_\infty^{\psi} \simeq \varLambda/(f(T)), \\
X_\infty^{\psi^*} \simeq \varLambda/(f^*(T)) & \;{\rm and}\; & 
Y_\infty^{\psi^*} \simeq \varLambda/(f^*(T)) \oplus \varLambda. \\
\end{array}
$$
\end{remark}

\section{\large Algorithms and numerical results}

We use the notation in the previous sections. 
First, we briefly review algorithms for computing Iwasawa invariants 
and exceptional pairs of $\mathbf{Q}(\sqrt{d},\zeta_p)$. 

We compute the following arithmetic special elements (cf.~~\cite{Aoki-Fukuda}, \cite[\S2]{Sumida6}, \cite[\S5]{Sumida7} and [\S3]\cite{Sumida8}): \\
(I) the generalized Bernoulli numbers modulo $p$, i.e., $\sum_{k=0}^{p-2} B_{k,\chi}t^k/k! \md p$, \\
(II)$_n$ the Iwasawa polynomial $g_{\psi}(T) \mod  p^{n+1}$, \\
(III)$_n$ the special cyclotomic unit $(c_n^{\psi})^{Y_n(T)}$ modulo a prime ideal $\mathfrak{L}_n$, \\
(IV)$_n$ the Gauss sum $g_0(N_{K_n/K_0}\mathfrak{L}_n)^{\psi^*}$ modulo a prime ideal $\mathfrak{L}_0^*$, where 
$\mathfrak{L}_n$ (resp.~$\mathfrak{L}_n^*$) is a prime ideal above $l=1+\kappa f_n$ 
(resp.~$l^*=1+\kappa^*(2 f_n l)$) of $K_n$. \\

In (I), we use the Fast Fourier Transform algorithm for multiplication of polynomials (see \cite[\S4.3.3]{Knuth}). 
In (II)$_n$, we use the approximate formula for $L(s,\psi)$ \cite[Theorem 5.11]{Washington}. 
In (III)$_n$, we define $Y_n(T)\in \mathbf{Z}[T]$ by $g_\psi(T) \mod p^{n+1}$ and compute $c=(c_n^{\psi})^{Y_n(T)} \mod \mathfrak{L}_n$ 
for an auxiliary prime number $l$. We check a criterion for $c$. 
If the criterion is satisfied, we obtain $\lambda_p(\psi)=0$ (Greenberg's conjecture) 
and an upper bound of $\nu_p(\psi)$ (see \cite[\S2]{Ichimura-Sumida2}). 
In (IV)$_n$, we compute $g=g_0(N_{K_n/K_0}\mathfrak{L}_n)^{\psi^*} \mod \mathfrak{L}_0^*$ 
for auxiliary prime $l^*$ by using the Fast Fourier Transform algorithm once again. 
We check a criterion for $g$. If the criterion is satisfied, we obtain an lower bound of $\nu_p(\psi)$. 
Finally we obtain the exact value of $\nu_p(\psi)$ if Greenberg's conjecture hold for $p$ and $\psi$ 
(see \cite[\S3]{Sumida6}). 

In \cite{Sumida6, Sumida7, Sumida8, Sumida10, Sumida12}, 
by using the above special elements, we efficiently computed irregular pairs and exceptional pairs 
for $|d|<200$ and $p<1,\!000,\!000$. 
By further computation, we obtain the following.

\begin{proposition}
For $|d|<200$ and $1,\!000,\!000<p<2,\!000,\!000$, all exceptional pairs $(p,\chi \omega^k)$ 
are given in {\rm Table 1}.
\label{prop:1}
\end{proposition}

\newpage

\begin{center}
{Table 1: Exceptional pairs for $|d|<200$ and $1,\!000,\!000<p<2,\!000,\!000$.}
\begin{tabular}{|ccc|ccc|}\hline
 & [$\nu$] &  &  & [$a_0$] &  \\ \hline
$p$ & $k$ & $d$ & $p$ & $k$ & $d$ \\ \hline 
1166927	& 871940 & 145 &  1243093 & 1052719 & -183 \\
1262483	& 816621 & -167 &  1580483 & 878413 & -116 \\
1321927	& 1087373& -35 & & & \\
1378691	& 497941 & -167& & & \\
1485761	& 566272 & 73& & & \\
1798813	& 984380 & 97& & & \\

\hline
 & [$b_0$] &  &  & [lmd] &  \\ \hline
$p$ & $k$ & $d$ & $p$ & $k$ & $d$ \\ \hline 
1128233	& 191315 & -67 & 1402147 & 944163 &  -87 \\
1140949	& 611954 & 56 & 1442899	 & 1057203 & -151 \\
1752481	& 557278 & 24 & 1744817	 & 928867 & -3 \\
\hline
\end{tabular}

\end{center}

In Figure 1, we compare actual numbers of exceptional pairs 
in the range $f_\chi <200$ and $200< p < x$ 
with the expected number 
$$
E_1(x)=123 
{\displaystyle \sum_{200 < p < x} \frac{p-3}{2}\left(\frac{1}{p}\right)^2},
$$
where $123$ is the number of the trivial character $\chi^0$ and quadratic characters $\chi$ with $f_\chi <200$.

\begin{center}
{Figure 1: Actual numbers and the expected number $(200<p<2,\!000,\!000)$.}
\includegraphics[width=130mm]{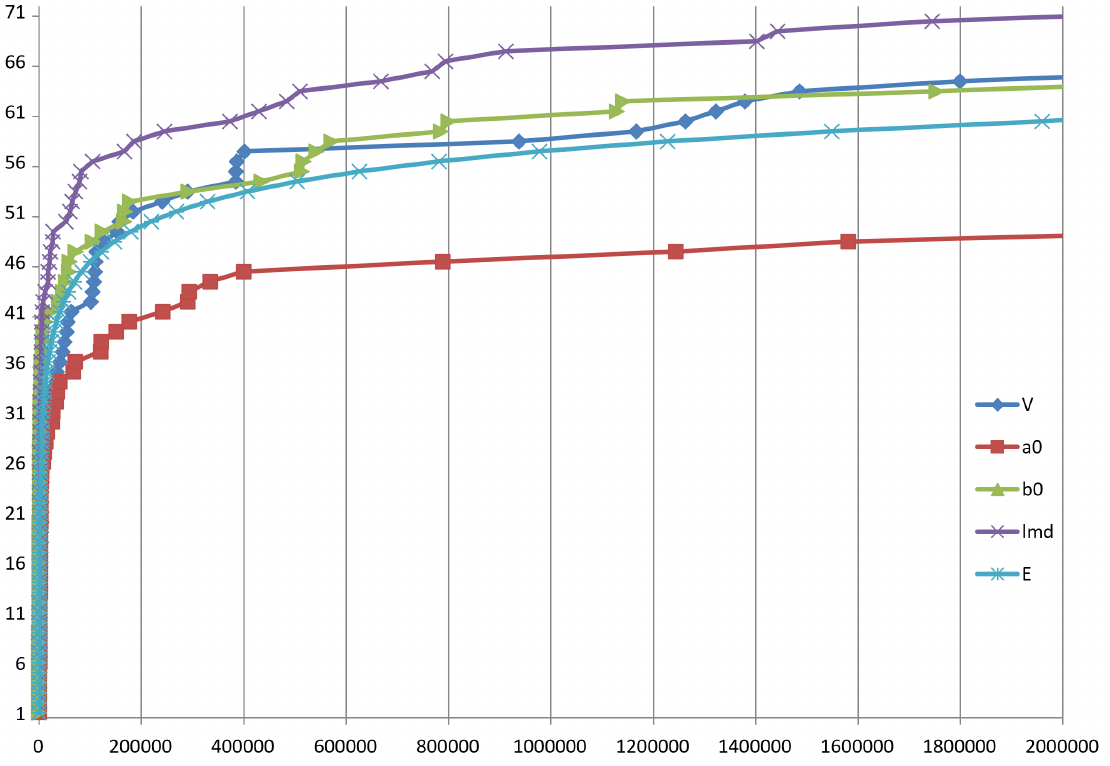}
\end{center}

In Figure 2, we compare actual numbers of exceptional pairs 
in the range $f_\chi <f$ and $200< p < 2,\!000,\!000$ 
with the expected number 
$$
E_2(n)=n {\displaystyle \sum_{200 < p < 2,000,000} \frac{p-3}{2}\left(\frac{1}{p}\right)^2},
$$
where $n=n(f)=\sharp \{\chi \, | \,\chi^2=\chi^0, \; f_\chi \leq f\}$. 
From our data, actual numbers still seem to be near the expected number.

\begin{center}
{Figure 2: Actual numbers and the expected number $(0\leq n(f) \leq 123)$.}
\includegraphics[width=130mm]{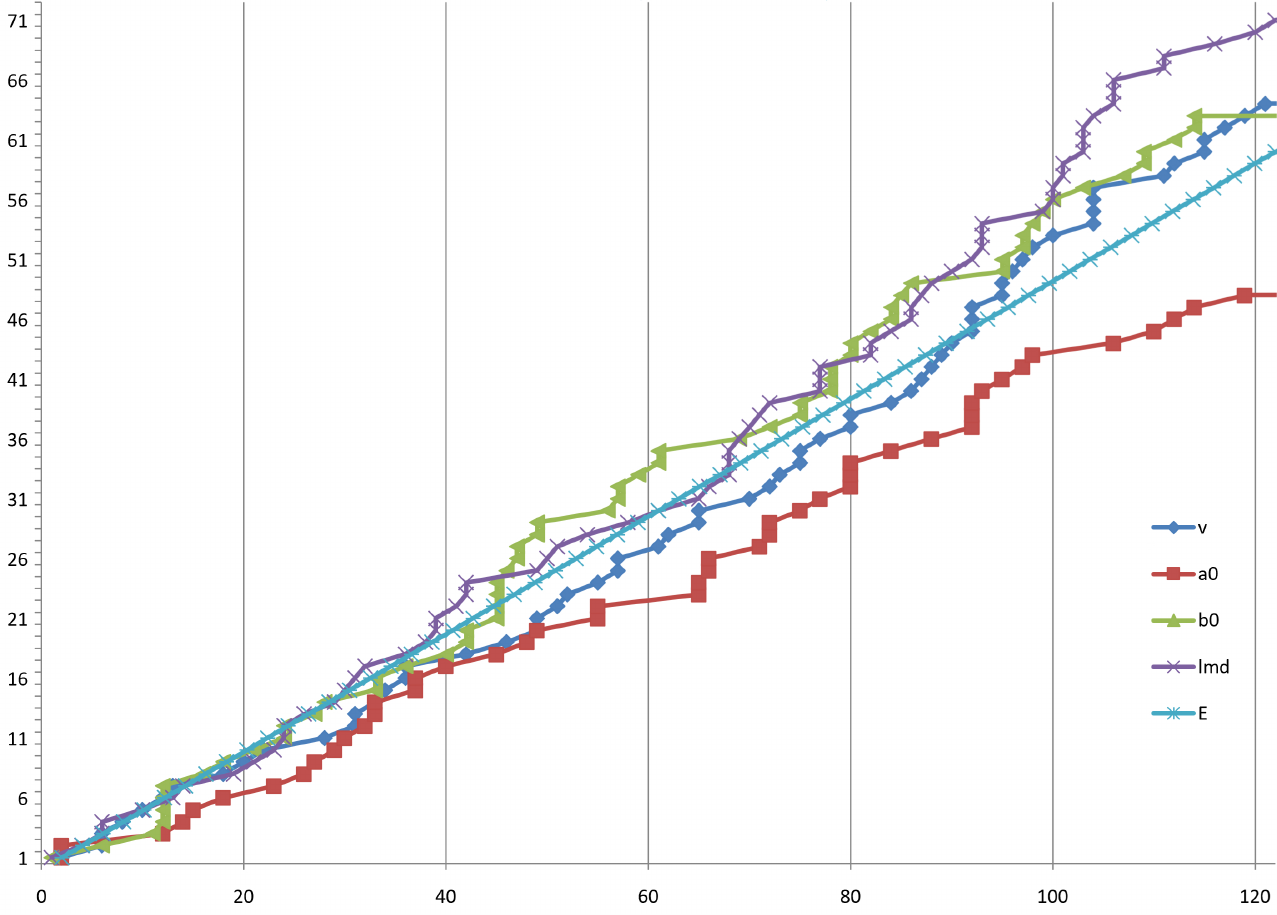}
\end{center}

In \cite{Sumida10, Sumida12}, we computed exceptional pairs 
for $|d|<10$ and $p<20,\!000,\!000$. 
The largest prime number among them is $5,\!911,\!877$ which was found in 2012 and 
satisfies the condition $[a_0]$ for $(k,d)=(1,\!629,\!992,5)$. 
By further computation, we obtain the following new exceptional pair which was found in 2024. 

\begin{proposition}
For $|d|<10$ and $6,\!000,\!000<p<30,\!000,\!000$, there is only one exceptional pair 
$(p,\chi \omega^k)$ which is given in {\rm Table 2}. 
\label{prop:2}
\end{proposition}

\begin{center}
{Table 2: The exceptional pair for $|d|<10$ and $6,\!000,\!000<p<30,\!000,\!000$.}
\begin{tabular}{|ccc|}\hline
     & [lmd] &  \\ \hline
 $p$ & $k$ & $d$ \\ \hline 
 28679999 & 2284521 & -8 \\
\hline
\end{tabular}
\end{center}

\begin{proposition}
In the range $|d|<200\; ($resp. $|d|<10)$ and $200<p<2,\!000,\!000\; ($resp. $10<p<30,\!000,\!000)$, 
every exceptional pair $(p,\psi)$ satisfies only one of conditions 
$[\nu]$, $[a_0]$, $[b_0]$ and $[{\rm lmd}]$. 
\label{prop:3}
\end{proposition}

For details of computation, see \cite{Sumida12}. 
The programs and further data have been available in our web page: 
https://hiro2.pm.tokushima-u.ac.jp/\~{}hiroki/\\
major/galois1-e.html. 
These data were obtained by seven personal computers with 76 cores (152 threads) for about five years. 
We computed twice for $|d|<200$ and $p<2,000,000$, and only once for $|d|<10$ and $p<30,000,000$. 

Though temporary hardware errors are rare, they could occur in long-term computation. 
We are afraid that they occur in the latter computation. 
However, these rare errors would not affect Proposition \ref{prop:2}, 
because the expected number is very small. 
In fact, 
$$
\sum_{x_0 < p <x_1} \dfrac{p-3}{2} \left(\dfrac{1}{p}\right)^2 
\approx \frac{1}{2} (\log \log 30,\!000,\!000-\log \log 6,\!000,\!000)=0.04907 \cdots 
$$
for $x_0=6,\!000,\!000$  and $x_1=30,\!000,\!000$. 
It might be hard to find a new exceptional pair for $|d|<10$, because 
$$
\sum_{x_0 < p <x_1} \dfrac{p-3}{2} \left(\dfrac{1}{p}\right)^2 \approx 
\frac{1}{2} (\log \log 100,\!000,\!000-\log \log 30,\!000,\!000)=0.03379 \cdots
$$
for $x_0=30,\!000,\!000$ and $x_1=100,\!000,\!000$.

\providecommand{\bysame}{\leavevmode\hbox to3em{\hrulefill}\thinspace}
\providecommand{\MR}{\relax\ifhmode\unskip\space\fi MR }
\providecommand{\MRhref}[2]{%
  \href{http://www.ams.org/mathscinet-getitem?mr=#1}{#2}
}
\providecommand{\href}[2]{#2}




\begin{thebibliography}{10}

\bibitem{Aoki-Fukuda}
M.~Aoki and T.~Fukuda, \emph{An algorithm for computing $p$-class groups of
  abelian number fields}, Lecture Notes in Computer Science \textbf{4076}
  (2006), 56--71.

\bibitem{Buhler1}
J.~Buhler, R.~Crandall, R.~Ernvall, T.~Mets{\"a}nkyl{\"a}, and A.~M.
  Shokrollahi, \emph{Irregular primes and cyclotomic invariants to $12$
  million}, J. Symbolic Comput. \textbf{31} (2001), 89--96.

\bibitem{Buhler2}
J.~Buhler and D.~Harvey, \emph{Irregular primes to 163 million}, Math. Comp.
  \textbf{80} (2011), 2435--2444.

\bibitem{Ferrero-Washington}
B.~Ferrero and L.~Washington, \emph{The {Iwasawa} invariant $\mu_p$ vanishes
  for abelian number fields}, Ann. of Math. \textbf{109} (1979), 377--395.

\bibitem{Franks}
C.~Franks, \emph{Classifying {$\Lambda$}-modules up to isomorphism and
  application to {Iwasawa} theory}, Ph D dissertation, Arizona State University
  (2011).

\bibitem{Greither}
C.~Greither, \emph{Class groups of abelian fields, and the main conjecture},
  Ann. Inst. Fourier $($Grenoble$)$ \textbf{42} (1992), 449--499.

\bibitem{Hart}
W.~Hart, D.~Harvey, and W.~Ong, \emph{Irregular primes to two billion}, Math.
  Comp. \textbf{86} (2017), 3031--3049.

\bibitem{Ichimura-Sumida2}
H.~Ichimura and H.~Sumida, \emph{On the {Iwasawa} invariants of certain real
  abelian fields {II}}, Internat. J. Math. \textbf{7} (1996), 721--744.

\bibitem{Iwasawa2}
K.~{Iwasawa}, \emph{Lectures on $p$-adic {L}-functions}, Ann. of Math. Stud.,
  vol.~74, Princeton Univ. Press: Princeton, N.J., 1972.

\bibitem{Iwasawa3}
\bysame, \emph{On {$\mathbf{Z}_l$}-extensions of algebraic number fields}, Ann.
  of Math.,(2) \textbf{98} (1973), 246--326.

\bibitem{Knuth}
E.~Knuth, \emph{The art of computer programming, vol. $2$: Seminumerical
  algorithms. $2$nd edition}, Addison-Wesley Publishing Co., Reading, Mass,
  1981.

\bibitem{Koike}
M.~Koike, \emph{On the isomorphism classes of {Iwasawa} modules associated to
  imaginary quadratic fields with $\lambda=2$}, J. Math. Sci. Univ. Tokyo
  \textbf{6} (1999), 371--396.

\bibitem{Kubota-Leopoldt}
T.~Kubota and H.W. Leopoldt, \emph{Eine $p$-adische {Theorie} der {Zetawerte},
  {I}. {Einf\"uhrung} der $p$-adischen {Dirichletschen $L$-Funktionen}}, J.
  reine angew. Math. \textbf{214/215} (1964), 328--339.

\bibitem{Mazur-Wiles}
B.~Mazur and A.~Wiles, \emph{Class fields of abelian extensions of
  $\mathbf{Q}$}, Invent. Math. \textbf{76} (1984), 179--330.

\bibitem{Murakami0}
K.~Murakami, \emph{On the isomorphism classes of {Iwasawa} modules with
  {{\(\lambda = 3\)}} and {{\(\mu = 0\)}}}, Osaka J. Math. \textbf{51} (2014),
  no.~4, 829--865 (English).

\bibitem{Murakami1}
\bysame, \emph{Isomorphism classes of modules over {Iwasawa} algebra with
  {{\(\lambda =4\)}}}, Tokyo J. Math. \textbf{39} (2016), no.~1, 101--132
  (English).

\bibitem{Murakami2}
\bysame, \emph{On a new invariant determining the isomorphism classes of
  {{\(\Lambda\)}}-modules with $\lambda = 3$}, Manuscr. Math. \textbf{164}
  (2021), 409--430.

\bibitem{Sumida}
H.~Sumida, \emph{Greenberg's conjecture and the {Iwasawa} polynomial}, J. Math.
  Soc. Japan \textbf{49} (1997), 689--711.

\bibitem{Sumida6}
H.~Sumida-Takahashi, \emph{Computation of {Iwasawa} invariants of certain real
  abelian fields}, J. Number Theory \textbf{105} (2004), 235--250.

\bibitem{Sumida7}
\bysame, \emph{The {Iwasawa} invariants and the higher {$K$-groups} associated
  to real quadratic fields}, Exp. Math. \textbf{14} (2005), 307--316.

\bibitem{Sumida8}
\bysame, \emph{Computation of the $p$-part of the ideal class group of certain
  real abelian fields}, Math. Comp. \textbf{76} (2007), 1059--1071.

\bibitem{Sumida10}
\bysame, \emph{Examples of the {Iwasawa} invariants and the higher {$K$-groups}
  associated to quadratic fields}, J. Math. Univ. Tokushima \textbf{41} (2007),
  33--41.

\bibitem{Sumida12}
\bysame, \emph{A generalized problem associated to the {Kummer-Vandiver}
  conjecture}, Arnold Math. J. \textbf{9} (2023), 381--391.

\bibitem{Sumida11}
H.~Sumida-Takahashi, N.~Furuya, and K.~Kitano, \emph{On the $l$-part of the
  class groups of imaginary cyclic fields of conductor $p$ and degree $2l^n$},
  J. Math. Tokushima Univ. \textbf{56} (2022), 1--10.

\bibitem{Sumida13}
\bysame, \emph{Greenberg's generalized conjecture and pairings of $p$-units in
  the $4p$-cyclotomic field}, J. Ramanujan Math, Soc. \textbf{39} (2024),
  79--90.

\bibitem{Washington}
L.~Washington, \emph{Introduction to cyclotomic fields. second edition},
  Graduate Texts in Math., vol.~83, Springer-Verlag: New York, 1997.

\end{thebibliography}
%

\end{document}